\documentclass[12pt,twoside]{amsart}
\usepackage{amsmath}
\usepackage{amsthm}
\usepackage{amsfonts}
\usepackage{amssymb}
\usepackage{latexsym}
\usepackage{mathrsfs}
\usepackage{amsmath}
\usepackage{amsthm}
\usepackage{amsfonts}
\usepackage{amssymb}
\usepackage{latexsym}
\usepackage{geometry}
\usepackage{dsfont}
\usepackage[dvips]{graphicx}
\usepackage{color}
\usepackage[all]{xy}

\date{}
\allowdisplaybreaks[4] \footskip=15pt
\renewcommand{\uppercasenonmath}[1]{}

\usepackage{graphicx,amssymb}
\usepackage[all]{xy}
\usepackage{amsmath}

\allowdisplaybreaks
\usepackage{amsthm}
\usepackage{color}

\theoremstyle{plain}
\newtheorem{theorem}{Theorem}[section]
\newtheorem{proposition}[theorem]{Proposition}
\newtheorem{lemma}[theorem]{Lemma}

\theoremstyle{definition}
\newtheorem{example}[theorem]{Example}
\newtheorem{definition}[theorem]{Definition}

\theoremstyle{definition}

\theoremstyle{remark}
\newtheorem{remark}[theorem]{Remark}

\newcommand{\Tor}{\mbox{\rm Tor}}

\newcommand{\Id}{\mathrm{Id}}

\def\Hom{{\rm Hom}}
\def\Ext{{\rm Ext}}
\def\Tor{{\rm Tor}}

\def\Ker{{\rm Ker}}

\def\Im{{\rm Im}}
\def\Coker{{\rm Coker}}

\def\Id{{\rm Id}}

\begin{document}
\begin{center}
{\large  \bf Some new module-theoretic characterizations of  $S$-coherent rings}

\vspace{0.5cm}  \ Xiaolei Zhang$^{a}$

{\footnotesize a.\ \ \ School of Mathematics and Statistics, Shandong University of Technology, Zibo 255049, China\\

E-mail: zxlrghj@163.com\\}
\end{center}

\bigskip
\centerline { \bf  Abstract}
\bigskip
\leftskip10truemm \rightskip10truemm \noindent

Let $R$ be a commutative ring with identity and $S$ a multiplicative subset of $R$. In this paper, we first introduce and study the notions of $s$-pure exact sequences and  $s$-absolutely pure modules which extend the classical notions of pure exact sequences and absolutely pure modules. And then, we give some new characterizations of  $S$-coherent rings in terms of   $s$-absolutely pure modules.
\vbox to 0.3cm{}

\leftskip0truemm \rightskip0truemm
\bigskip

\section{introduction and Preliminary}

Throughout this paper, $R$  is always a commutative ring with identity, all modules are unitary and $S$  is always a multiplicative subset of $R$, that is, $1\in S$ and $s_1s_2\in S$ for any $s_1\in S, s_2\in S$.


In 2002,  Anderson and Dumitrescu \cite{ad02} introduced \emph{$S$-Noetherian rings} $R$ in which every ideal $I$ of $R$ is $S$-finite, that is, there exists a finitely generated sub-ideal $K$ of $I$ such that  $sI\subseteq K$.  Cohen's Theorem, Eakin-Nagata Theorem and Hilbert Basis Theorem for $S$-Noetherian rings are given in \cite{ad02}. In 2018, Bennis and Hajoui \cite{bh18} introduced the notions of $S$-finitely presented modules and $S$-coherent modules which can be seen as $S$-versions of finitely presented modules and coherent modules. An $R$-module $M$ is said to be \emph{$S$-finitely presented} provided  there exists
an exact sequence of $R$-modules $$0 \rightarrow K\rightarrow F\rightarrow M\rightarrow 0,$$ where $K$ is
$S$-finite and $F$ is  finitely generated free. An $R$-module $M$ is said to be \emph{$S$-coherent} if it is finitely generated and every finitely generated submodule of $M$ is $S$-finitely presented. A ring $R$ is called an \emph{$S$-coherent ring} if $R$ itself is an $S$-coherent $R$-module, that is, every finitely generated ideal of $R$ is $S$-finitely presented.

Recently, the author, Qi and Zhao \cite{qwz23} gave an $S$-version of Chase Theorem in terms of $S$-flat modules, i.e., modules $M$ satisfying $M_S$ is a flat $R_S$-module. We give a new $S$-version of Chase Theorem in this paper (see Theorem \ref{newchase}).

It is well-known that a ring $R$ is a coherent ring if and only if every pure quotient of an absolutely pure module is absolutely pure, if and only if every direct limits of absolutely pure modules is absolutely pure (see \cite[Theorem 3.2]{S70}). It is also well-known that  a ring $R$ is  coherent if and only if $\Hom_R(E,I)$ is  flat for any absolutely pure module $E$ and injective module $I$, if and only if $\Hom_R(\Hom_R(F,I_1),I_2)$ is flat  for any  flat module $F$     and any injective modules $I_1$ and $I_2$ (see \cite[Lemma 4]{CD93}). A natural question is that

\textbf{Question:}
	How to characterize $S$-coherent rings in terms of a certain $S$-version  of absolutely pure modules?\\
Actually, we show that a ring $R$ is an $S$-coherent ring if and only if every ($s$-)pure quotient of an ($s$-)absolutely pure module is $s$-absolutely pure, if and only if every direct limits of absolutely pure modules is $s$-absolutely pure (see Theorem \ref{s-d-d-non}). We also show that a ring $R$ is an $S$-coherent ring if and only if   $\Hom_R(E,I)$ is $s$-flat  for any    $s$-absolutely pure module $E$  and any injective module $I$, if and only if    if $I$  is an injective cogenerator, then   $\Hom_R(E,I)$ is $s$-flat  for any  $s$-absolutely pure module $E$, if and only if  $\Hom_R(\Hom_R(F,I_1),I_2)$ is $s$-flat  for any  $s$-flat module $F$     and any injective modules $I_1$ and $I_2$, if and only if if $I_1$ and $I_2$ are injective cogenerators, then $\Hom_R(\Hom_R(F,I_1),I_2)$ is $s$-flat    for any   $s$-flat module $F$ (see Theorem \ref{phi-coh-fp}).

\section{$s$-pure exact sequences}

Recall from \cite{zwz21} that an $R$-module $T$ is called a $u$-$S$-torsion $($abbreviates uniformly $S$-torsion$)$ module  provided that there exists an element $s\in S$ such that $sT=0$.  An $R$-sequence  $M\xrightarrow{f} N\xrightarrow{g} L$is said to be  $u$-$S$-exact (at $N$) provided that there is an element $s\in S$ such that $s\Ker(g)\subseteq \Im(f)$ and $s\Im(f)\subseteq \Ker(g)$. A long $R$-sequence is said to be  $u$-$S$-exact  if it  $u$-$S$-exact at any place. An $R$-homomorphism $M\xrightarrow{f} N$ is called a $u$-$S$-monomorphism (resp., $u$-$S$-epimorphism, $u$-$S$-isomorphism) if $0\rightarrow M\xrightarrow{f} N$ (resp., $M\xrightarrow{f} N\rightarrow 0$, $0\rightarrow M\xrightarrow{f} N\rightarrow 0$) is $u$-$S$-exact.

Recall from \cite{R79} that an exact sequence $0\rightarrow A\rightarrow B\rightarrow C\rightarrow 0$ is said to be pure provided that for any finitely presented $R$-module $M$, the induced sequence $0\rightarrow M\otimes_RA\rightarrow M\otimes_RB\rightarrow M\otimes_RC\rightarrow 0$ is also exact. Now we introduce the uniformly $S$-version of pure exact sequences.

\begin{definition}\label{def-s-f} Let $R$ be a ring and $S$ a multiplicative subset of $R$.
	A short exact sequence $0\rightarrow A\rightarrow B\rightarrow C\rightarrow 0$ is said to be \emph{$s$-pure} provided that for any finitely presented $R$-module $M$, the induced sequence $$0\rightarrow M\otimes_RA\rightarrow M\otimes_RB\rightarrow M\otimes_RC\rightarrow 0$$ is $u$-$S$-exact.
\end{definition}

Obviously, any  pure exact sequence is $s$-pure.  In \cite[34.5]{w}, there are many characterizations of pure exact sequences. The next result generalizes some of those characterizations to $s$-pure exact sequences.

\begin{theorem}\label{c-s-pure}
	Let $0\rightarrow A\xrightarrow{f} B\xrightarrow{f'} C\rightarrow 0$  be a  short exact sequence of $R$-modules.
	Then the following statements are equivalent:
	\begin{enumerate}
		\item $0\rightarrow A\xrightarrow{f} B\xrightarrow{f'} C\rightarrow 0$ is $s$-pure.
		\item If a system of equations $f(a_i)=\sum\limits^m_{j=1}r_{ij}x_j\ (i=1,\dots,n)$ with $r_{ij} \in R$ and unknowns $x_1,\dots, x_m$ has a solution in $B$, then  there exists an element $s\in S$ such that the system of equations  $sa_i=\sum\limits^m_{j=1}r_{ij}x_j\ (i=1,\dots,n)$ is solvable in $A$.
		\item  For any given commutative diagram with $F$ finitely generated free and $K$ a finitely generated submodule of $F$, there exists an  $R$-homomorphism $\eta:F\rightarrow A$ and $s\in S$ such that $s\alpha=\eta i$:
		$$\xymatrix@R=20pt@C=25pt{
			0\ar[r] &K\ar[d]_{\alpha}\ar[r]^{i}&F\ar@{.>}[ld]^{\eta}\ar[d]^{\beta}\\
			& A\ar[r]_{f} &B\\};$$
		\item For any finitely presented $R$-module $N$, the induced sequence $0\rightarrow\Hom_R(N,A)\rightarrow \Hom_R(N,B)\rightarrow \Hom_R(N,C)\rightarrow 0$ is  $u$-$S$-exact with respect to some $s\in S$.
	\end{enumerate}
\end{theorem}
\begin{proof} $(1)\Rightarrow(2)$:  Let $M$ be a finitely presented $R$-module. Then there is an exact sequence $0\rightarrow K\rightarrow F\rightarrow M\rightarrow0$ with $F$ finitely generated free $R$-module and $K$ a finitely generated submodule of $F$. Then  $0\rightarrow M\otimes_RA\xrightarrow{1\otimes f} M\otimes_RB\rightarrow M\otimes_RC\rightarrow 0$ is $u$-$S$-exact by $(1)$.  So there is an element $s\in S$ such that $s\Ker(1_M\otimes f)=0$. Hence $s\Ker(1_{F/K}\otimes f)=0$. Now assume that there exists $b_j\in B$ such that $f(a_i)=\sum\limits^m_{j=1}r_{ij}b_j$ for any $j=1,\dots,m$. Assume that $\{e_1,\dots,e_n\}$ is the basis of  $F$ and suppose $K$ is generated by $m$ elements $\{\sum\limits^n_{i=1}r_{ij}e_i\mid j=1,\dots,m\}$. Then, $M=F/K$ is generated by $\{e_1+K,\dots,e_n+K\}$. Note that $\sum\limits^n_{i=1}r_{ij}(e_i+K)=\sum\limits^n_{i=1}r_{ij}e_i+K=0+K$ in $F/K$. Hence, we have $$\sum\limits^n_{i=1}((e_i+K)\otimes f(a_i))=\sum\limits^n_{i=1}((e_i+K)\otimes (\sum\limits^m_{j=1}r_{ij}b_j))=\sum\limits^m_{j=1}((\sum\limits^n_{i=1}r_{ij}(e_i+K))\otimes b_j)=0$$
	in $F/K\otimes B$. And so $\sum\limits^n_{i=1}((e_i+K)\otimes a_i)\in \Ker(1_{F/K}\otimes f)$.
	Hence, $s\sum\limits^n_{i=1}((e_i+K)\otimes a_i)=\sum\limits^n_{i=1}((e_i+K)\otimes sa_i)=0$ in $F/K\otimes_RA$. By \cite[Chapter I, Lemma 6.1]{FS01}, there exists $d_j\in A$ and $t_{ij}\in R$ such that $sa_i=\sum\limits^t_{k=1}l_{ik}d_k$ and $\sum\limits^n_{i=1}l_{ik}(e_i+K)=0$, and so  $\sum\limits^n_{i=1}l_{ik}e_i\in K$. Then there exists $t_{jk}\in R$ such that $\sum\limits^n_{i=1}l_{ik}e_i=\sum\limits^m_{j=1}t_{jk}(\sum\limits^n_{i=1}r_{ij}e_i)=\sum\limits^n_{i=1}(\sum\limits^m_{j=1}(t_{jk}r_{ij})e_i)$. Since $F$ is free, we have $l_{ik}=\sum\limits^m_{j=1}r_{ij}t_{jk}$. Hence
	$$sa_i=\sum\limits^t_{k=1}l_{ik}d_k=\sum\limits^t_{k=1}(\sum\limits^m_{j=1}r_{ij}t_{jk})d_k=
	\sum\limits^m_{j=1}r_{ij}(\sum\limits^t_{k=1}t_{jk}d_k)$$
	with $\sum\limits^t_{k=1}t_{jk}d_k\in A$. That is, $sa_i=\sum\limits^m_{j=1}r_{ij}x_j$ is solvable in $A$.

	$(2)\Rightarrow(1)$: Let $M$ be a finitely presented $R$-module. Then we have a exact sequence
	$M\otimes_RA\xrightarrow{1\otimes f} M\otimes_RB\rightarrow M\otimes_RC\rightarrow 0$ . We will show that $\Ker(1\otimes f)$ is  $u$-$S$-torsion. Let $\{\sum\limits^{n_\lambda}_{i=1}u^\lambda_i\otimes a^\lambda_i \mid \lambda\in \Lambda\}$ be the generators of $\Ker(1\otimes f)$. Then  $\sum\limits^{n_\lambda}_{i=1}u^\lambda_i\otimes f(a^\lambda_i)=0$ in $M\otimes_RB$ for each $\lambda\in \Lambda$. By \cite[Chapter I, Lemma 6.1]{FS01}, there exists $r^\lambda_{ij}\in R$ and $b^\lambda_j\in B$ such that $f(a^\lambda_i)=\sum\limits^{m_\lambda}_{j=1}r^\lambda_{ij}b^\lambda_j$ and $\sum\limits^{n_\lambda}_{i=1}u^\lambda_ir^\lambda_{ij}=0$ for each $\lambda\in \Lambda$. So $sa^\lambda_i=\sum\limits^{m_\lambda}_{j=1}r^\lambda_{ij}x^\lambda_j$  have a solution, say $a^\lambda_j$ in $A$ by (2). Then $$s(\sum\limits^{n_\lambda}_{i=1}u^\lambda_i\otimes a^\lambda_i)=\sum\limits^{n_\lambda}_{i=1}u^\lambda_i\otimes sa^\lambda_i=\sum\limits^{n_\lambda}_{i=1}u^\lambda_i\otimes (\sum\limits^{m_\lambda}_{j=1}r^\lambda_{ij}a^\lambda_j)=\sum\limits^{m_\lambda}_{j=1}((\sum\limits^{n_\lambda}_{i=1}r^\lambda_{ij}u^\lambda_i)\otimes a^\lambda_j)=0$$ for each $\lambda\in \Lambda$. Hence $s\Ker(1\otimes f)=0$, and $0\rightarrow M\otimes_RA\rightarrow M\otimes_RB\rightarrow M\otimes_RC\rightarrow 0$ is  $u$-$S$-exact.
	
	$(2)\Rightarrow(3)$:  Let  $\{e_1,\dots,e_n\}$ be the basis of the finitely generated free $R$-module $F$. Suppose $K$ is generated by $\{y_i=\sum\limits^{m}_{j=1}r_{ij}e_j\mid i=1,\dots,m\}$. Set $\beta(e_j)=b_j$ and $\alpha(y_i)=a_i$ for each $i$ and $j$. Then $f(a_i)=\sum\limits^{m}_{j=1}r_{ij}b_j$. By $(2)$, we have
	$sa_i=\sum\limits^m_{j=1}r_{ij}d_j$ for some $d_j\in A$. Let $\eta:F\rightarrow A$ be the $R$-homomorphism satisfying $\eta(e_j)=d_j$. Then $\eta i(y_i)=\eta i(\sum\limits^{m}_{j=1}r_{ij}e_j)=\sum\limits^{m}_{j=1}r_{ij}\eta(e_j)=\sum\limits^{m}_{j=1}r_{ij}d_j=sa_i=s\alpha(y_i)$,
	and so we have $s\alpha=\eta i$.
	
	$(3)\Rightarrow(4)$: Let $\gamma\in\Hom_R(N,C)$. Then there are $R$-homomorphisms $\delta:N\rightarrow B$ and $\alpha:K\rightarrow A$ such that the  following  diagram is commutative with rows exact:
	$$\xymatrix@R=20pt@C=30pt{
		0\ar[r] &K\ar[d]_{\alpha}\ar[r]^{i}&F\ar[r]^{\pi} \ar[d]^{\beta}& N\ar[d]^{\gamma}\ar[r]&0\\
		0\ar[r] & A\ar[r]_{f} &B\ar[r]_{f'}  &C\ar[r] & 0\\}$$
Let $s\in S$. Then we have following  diagram is commutative with rows exact:
	$$\xymatrix@R=20pt@C=30pt{
		0\ar[r] &K\ar[d]_{s\alpha}\ar[r]^{i}&\ar@{.>}[ld]^{\eta}F\ar[r]^{\pi} \ar[d]^{s\beta}& N\ar@{.>}[ld]^{\delta}\ar[d]^{s\gamma}\ar[r]&0\\
		0\ar[r] & A\ar[r]_{f} &B\ar[r]_{f'}  &C\ar[r] & 0\\}$$
	By (3), there exists an homomorphism $\eta:F\rightarrow A$ such that $s\alpha=\eta i_K$. It follows by  \cite[Exercise 1.60]{fk16}, there is an $R$-homomorphism $\delta:N\rightarrow B$ such that $s\gamma=f'\delta$. Consequently, one can verify the  $R$-sequence $0\rightarrow \Hom_R(N,A)\rightarrow \Hom_R(N,B)\rightarrow \Hom_R(N,C)\rightarrow 0$ is $u$-$S$-exact  with respect to  $s$.

	$(4)\Rightarrow(2)$:  Suppose that  $f(a_i)=\sum\limits^m_{j=1}r_{ij}b_j\ (i=1,\dots,n)$ with $a_i\in A$, $b_j\in B$ and  $r_{ij}\in R$. Let $F_0$ be a free module with a basis $\{e_1,\dots,e_m\}$ and  $F_1$ a free module with basis $\{e'_1,\dots,e'_n\}$. Then there are $R$-homomorphisms $\tau: F_0\rightarrow B$ and $\sigma: F_1\rightarrow \Im(f)$ satisfying $\tau(e_j)=b_j$ and $\sigma(e'_i)=f(a_i)$ for each $i,j$. Define an $R$-homomorphism $h:F_1\rightarrow F_0$ by $h(e'_i)=\sum\limits^m_{j=1}r_{ij}e_j$ for each $i$. Then $\tau h(e'_i)=\sum\limits^m_{j=1}r_{ij}\tau(e_j)=\sum\limits^m_{j=1}r_{ij}b_j=f(a_i)=\sigma(e'_i)$. Set $N=\Coker(h)$. Then $N$ is finitely presented. Thus there exists a homomorphism $\phi: N\rightarrow \Coker(f)$ such that the following diagram commutative:
	$$\xymatrix@R=20pt@C=30pt{
		&F_1\ar[d]_{\sigma}\ar[r]^{h}&F_0\ar[r]^{g} \ar[d]^{\tau}& N\ar@{.>}[d]^{\phi}\ar[r]&0\\
		0\ar[r] & A\ar[r]_{i} &B\ar[r]_{\pi}  &C\ar[r] & 0\\}$$
	Note that the induced sequence $$\Hom_R(N,B)\rightarrow \Hom_R(N,C)\rightarrow 0$$ is $u$-$S$-exact with respect to $s$ by (4). Hence there exists a homomorphism $\delta: N\rightarrow C$ such that $s\phi=\pi\delta$.
	Consider the following commutative diagram:
	$$\xymatrix@R=20pt@C=30pt{
		&F_1\ar[d]_{s\sigma}\ar[r]^{h}&F_0\ar@{.>}[ld]^{\eta}\ar[r]^{g} \ar[d]^{s\tau}& N\ar[ld]^{\delta}\ar[d]^{s\phi}\ar[r]&0\\
		0\ar[r] & A\ar[r]_{i} &B\ar[r]_{\pi}  &C\ar[r] & 0\\}$$
	We claim that there exists a homomorphism $\eta:F_0\rightarrow A$ such that $\eta f=s\sigma$. Indeed, since $\pi\delta g=s\phi g=\pi \tau$, we have $\Im(s\tau-\delta g)\subseteq \Ker(\pi)=\Im(f)$. Define $\eta:F_0\rightarrow \Im(f)$ to be a homomorphism satisfying $\eta(e_i)=s\tau(e_i)-\delta g(e_i)$ for each $i$. So for each $e'_i\in F_1$, we have $\eta f(e'_i)=s\tau f(e'_i)-\delta g f(e'_i)=s\tau f(e'_i)$. Thus $i(s\sigma)=si\sigma=s\tau f=i\eta f$. Therefore, $\eta f=s\sigma$. Hence $sf(a_i)=s\sigma(e'_i)=\eta f(e'_i)=\eta (\sum\limits^m_{j=1}r_{ij}e_j)=\sum\limits^m_{j=1}r_{ij}\eta (e_j)$ with $\eta (e_j)\in \Im(f)$. So we have $sa_i=st'_1f(a_i) =\sum\limits^m_{j=1}r_{ij}t'_1\eta (e_j)$ with $t'_1\eta (e_j)\in A$ for each $i$.
\end{proof}

Recall from \cite{qwz23} that an $R$-modules $M$ is said to be $S$-flat if   $M_S$ is a flat $R_S$-module.
\begin{definition}
An $R$-modules $F$ is said to be \emph{$s$-flat} if  for any finitely presented $R$-module $M$, there exists $s\in S$ such that $s\Tor_1^R(M,F)=0$.
\end{definition}

Certainly,  every $s$-flat module is $S$-flat. Indeed,  let $N$ be a finitely presented $R_S$-module. Then $N\cong M_S$ for some finitely presented $R$-module $M$. So $\Tor_1^{R_S}(N,F_S)\cong \Tor_1^{R}(M,F)_S=0$ as $s\Tor_1^R(M,F)=0$ for some $s\in S$. However, the converse is not true in general.

\begin{example}
	Let $D$ be a valuation domain whose valuation group is $\mathbb{Z}\times\mathbb{Z}$ with lexicographical order and $v: R\rightarrow G$ is the valuation map. Let $x,y\in R$ with $v(x)=(0,1)$ and $v(y)=(1,0)$. Then for any $i\geq 1$, there exists $b_i\in R$ such that $y=b_ix^i$. Let $S=\{x^i\mid i\geq 0\}.$
	Set $$M=\bigoplus\limits_{i\geq 1}R/x^iR.$$	
	Then $M_S=0$, and so $M$ is an $S$-flat $R$-module. We claim that $M$ is not $s$-flat. Indeed, let $N=R/yR$. Then $N$ is finitely presented. However, $$\Tor_1^R(N,M)\cong \bigoplus\limits_{i\geq 1}\Tor_1^R(R/yR,R/x^iR)\cong \bigoplus\limits_{i\geq 1}yR/yx^iR.$$
	So for any $x^i\in S$, we have $x^i\Tor_1^R(N,M)\not=0$. Hence, $M$ is not $s$-flat.
\end{example}

\begin{lemma}\label{spq-f}
Any $s$-pure quotient module of $s$-flat module  is $s$-flat.
\end{lemma}
\begin{proof}
	Let $0\rightarrow A\rightarrow B\rightarrow C\rightarrow 0$ be an  $s$-pure exact sequence with $B$ $s$-flat. Let $M$ be a finitely presented $R$-module. Then the result follows by the induced exact sequence $$\Tor_1^R(M,B)\rightarrow\Tor_1^R(M,C)\rightarrow M\otimes_RA\rightarrow M\otimes_RB\rightarrow M\otimes_RC\rightarrow0.$$
\end{proof}

\begin{proposition} An $R$-module $F$ is $s$-flat if and only if any exact sequence $0\rightarrow A\rightarrow B\rightarrow F\rightarrow 0$ is  $s$-pure.
\end{proposition}
\begin{proof}  Let $M$ be a finitely presented $R$-module. Suppose $F$ is $s$-flat.  Then $\Tor_1^R(M,F)$ is uniformly $S$-torsion. The result follows by the induced exact sequence $$\Tor_1^R(M,F)\rightarrow M\otimes_RA\rightarrow M\otimes_RB\rightarrow M\otimes_RF\rightarrow0.$$
	
On the other hand,	let $0\rightarrow A\rightarrow P\rightarrow F\rightarrow0$ be an exact sequence with $P$ a projective $R$-module. Then the result follows by $$0\rightarrow\Tor_1^R(M,F)\rightarrow M\otimes_RA\rightarrow M\otimes_RP\rightarrow M\otimes_RF\rightarrow0.$$
\end{proof}

\section{$s$-absolutely pure modules}
Recall  thatAn $R$-module $M$ is said to be absolutely pure provided that $\Ext_R^1(N,M)=0$  for any finitely presented $R$-module $N$.
\begin{definition}\label{def-s-f} Let $R$ be a ring, $S$ a multiplicative subset of $R$.
 An $R$-module $M$ is said to be $s$-absolutely pure provided that for any finitely presented $R$-module $N$, there is $s\in S$ such that $s\Ext_R^1(N,M)=0$.
\end{definition}

Recall from  \cite[Definition 2.1]{zqcspsss} that a short $u$-$S$-exact sequence  $0\rightarrow A\xrightarrow{f} B\xrightarrow{g} C\rightarrow 0$ is said to be $u$-$S$-split provided that there are  $s\in S$ and an  $R$-homomorphism $f':B\rightarrow A$ such that $f'f(a)=sa$ for any $a\in A$, that is, $f'f=s\Id_A$. Next, we will give a characterization of $s$-absolutely pure modules.
\begin{theorem}\label{c-s-abp}
	Let $R$ be a ring, $S$ a multiplicative subset of $R$ and $E$ an $R$-module.
	Then the following statements are equivalent:
	\begin{enumerate}
		\item $E$ is $s$-absolutely pure;
		\item  any short exact sequence $0\rightarrow E\rightarrow B\rightarrow C\rightarrow 0$ beginning with $E$ is $s$-pure;
		\item  $E$ is a $s$-pure submodule in every injective module containing $E$;
		\item  $E$ is a $s$-pure submodule in its injective envelope;
		\item  if  $P$ is finitely generated projective, $K$ is a finitely generated submodule of $P$ and  $f:K\rightarrow E$ is an  $R$-homomorphism, then there is an $R$-homomorphism $g:P\rightarrow E$ such that $sf=gi$ for some $s\in S$ with $i:K\rightarrow P$  the natural embedding map.
	\end{enumerate}
\end{theorem}

\begin{proof} $ (2)\Rightarrow (3)\Rightarrow (4)$: Trivial.
	
	$(4)\Rightarrow(1)$: Let $I$ be the injective envelope of $E$. Then we have an $s$-pure exact sequence $0\rightarrow E\rightarrow I\rightarrow L\rightarrow 0$ by (4). Then, by  Theorem \ref{c-s-pure}, there is an element $s\in S$ such that $0\rightarrow\Hom_R(N,E)\rightarrow \Hom_R(N,I)\rightarrow \Hom_R(N,L)\rightarrow 0$  is $u$-$S$-exact with respect to $s$  for any finitely presented $R$-module $N$. Since  $0\rightarrow\Hom_R(N,E)\rightarrow \Hom_R(N,I)\rightarrow \Hom_R(N,L)\rightarrow \Ext_R^1(N,E)\rightarrow 0$ is exact. Hence $\Ext_R^1(N,E)$ is  $u$-$S$-torsion with respect to $s$ for any finitely presented $R$-module $N$.
	
	$(1)\Rightarrow(2)$:  Let  $N$ be a finitely presented $R$-module and $0\rightarrow E\rightarrow B\rightarrow C\rightarrow 0$ an exact sequence. Then there is an exact sequence $0\rightarrow\Hom_R(N,E)\rightarrow \Hom_R(N,B)\rightarrow \Hom_R(N,C)\rightarrow \Ext_R^1(N,E)$  for any finitely presented $R$-module $N$. By (1), $0\rightarrow\Hom_R(N,E)\rightarrow \Hom_R(N,B)\rightarrow \Hom_R(N,C)\rightarrow  0$  is $u$-$S$-exact  for any finitely presented $R$-module $N$.
	Hence $0\rightarrow E\rightarrow B\rightarrow C\rightarrow 0$ is $s$-pure by  Theorem \ref{c-s-pure}.
	
	$(1)\Rightarrow(5)$:   Considering the exact sequence $0\rightarrow K\xrightarrow{i} P\rightarrow P/K\rightarrow 0$, we have the following exact sequence $ \Hom_R(P,E)\xrightarrow{i_{\ast}} \Hom_R(K,E)\rightarrow \Ext_R^1(P/K,E)\rightarrow 0$. Since $P/K$ is finitely presented, $\Ext_R^1(P/K,E)$ is $u$-$S$-torsion with respect to some $s\in S$ by $(1)$. Hence $i_{\ast}$ is a $u$-$S$-epimorphism, and so $s\Hom_R(K,E)\subseteq \Im(i_{\ast})$. Let $f:K\rightarrow E$ be an $R$-homomorphism. Then there is an $R$-homomorphism $g:P\rightarrow E$ such that $sf=gi$.
	
	$(5)\Rightarrow(1)$:  Let $N$ be a finitely presented $R$-module.  Then we have an exact sequence $0\rightarrow K\xrightarrow{i} P\rightarrow N\rightarrow 0$ where $P$ is finitely generated projective and $K$ is finitely generated. Consider the following exact sequence $ \Hom_R(P,E)\xrightarrow{i_{\ast}} \Hom_R(K,E)\rightarrow \Ext_R^1(N,E)\rightarrow 0$.  By $(5)$, we have  $s\Hom_R(K,E)\subseteq \Im(i_{\ast})$. Hence $\Ext_R^1(N,E)$ is  $u$-$S$-torsion with respect to $s$.
\end{proof}

\begin{lemma}\label{inj-cog}\cite[lemma 4.2]{zuscoh}
	Let $E$ be an injective cogenerator. Then the following statements are equivalent:
	\begin{enumerate}
		\item $T$ is  uniformly $S$-torsion with respect to $s$;
		\item   $\Hom_R(T,E)$  is  uniformly $S$-torsion with respect to $s$.
	\end{enumerate}
\end{lemma}
\begin{proposition}\label{flat-FP-injective}
	Let $R$ be a ring and $S$ be a multiplicative subset of $R$. Then the following statements are equivalent:
	\begin{enumerate}
		\item $F$ is $s$-flat;
		\item   $\Hom_R(F,E)$ is $s$-absolutely pure for any injective module $E$;
		\item  if $E$ is an injective cogenerator, then $\Hom_R(F,E)$ is $s$-absolutely pure.
	\end{enumerate}
\end{proposition}
\begin{proof}

	$(1)\Rightarrow (2)$: Let $M$ be a finitely presented $R$-module and $E$ be an injective $R$-module. Since $F$ is $s$-flat,   $\Tor_1^R(M,F)$  is uniformly $S$-torsion. Thus $\Ext_R^1(M,\Hom_R(F,E))\cong\Hom_R(\Tor_1^R(M,F),E)$ is also uniformly $S$-torsion by \cite[Lemma 4.2]{QKWCZ21}.  Thus $\Hom_R(F,E)$ is $s$-absolutely pure.
	
	$(2)\Rightarrow (3)$: Trivial.

	$(3)\Rightarrow (1)$:  Let  $E$ be an injective cogenerator. Since $\Hom_R(F,E)$ is $s$-absolutely pure, for any finitely presented $R$-module $M$ there exists  $s\in S$ such that  $$\Hom_R(\Tor_1^R(M,F),E)\cong \Ext_R^1(M,\Hom_R(F,E))$$ is  uniformly $S$-torsion with respect to $s$. Since $E$ is an injective cogenerator, $\Tor_1^R(M,F)$ is uniformly $S$-torsion with respect to $s$  by Lemma \ref{inj-cog}.
\end{proof}

\begin{proposition}\label{spsub} Any $s$-pure submodule of an $s$-absolutely pure module is $s$-absolutely pure.
\end{proposition}
\begin{proof} Let $M_2$ be an $s$-absolutely pure module and $M_1$ be an $s$-pure submodule of $M_2.$ Let $K$ be a finitely generated submodule of a free module $F$. Let $\alpha:K\rightarrow M_1$ be an $R$-homomorphism. Consider the following commutative diagram
		$$\xymatrix@R=20pt@C=25pt{
			0\ar[r] &K\ar[d]_{s_1\alpha}\ar[r]^{i}&F\ar@{.>}[ld]^{\eta}\ar@{.>}[d]^{\beta}\\
			0\ar[r] & M_1\ar[r]_{f} &M_2\\}$$
Since $M_2$ is $s$-absolutely pure, there is $s_1\in S$ and an $R$-homomorphism $\beta:F\rightarrow M_2 $ such that $s_1f\alpha=\beta i$ by Theorem \ref{c-s-abp}. Since $f$ is $s$-pure ,  there is $s_2\in S$ and an $R$-homomorphism $\eta:F\rightarrow M_1 $ such that $s_1s_2\alpha=\eta i$ by Theorem \ref{c-s-pure}. Hence $M_1$ is $s$-absolutely pure by Theorem \ref{c-s-abp}.
\end{proof}

\begin{remark} The author in \cite{bb24} introduced the notion of $S$-FP-injective modules, i.e., $R$-modules $M$ satisfying $\Ext^1_R(C_S, M) = 0$ for every finitely presented $R$-module $C.$	Trivially,  $S$-FP-injective modules is closed under direct products.
\end{remark}

However, the following example shows that the class of $s$-absolutely pure modules are not closed under direct products and direct sums. So $s$-absolutely pure modules are different with $S$-FP-injective modules.
\begin{example}
	Let $\mathbb{Z}$ be the ring of integers, $p$ a prime in $\mathbb{Z}$ and  $S=\{p^n\mid n\geq 0\}$.  Set  $E_n=\mathbb{Z}/\langle p^n\rangle$. Then $E_n$ is an $s$-absolutely pure $\mathbb{Z}$-module.  Set $E=\prod\limits_{n=1}^{\infty}E_n$. Let $q\not=p$ be another prime  in $\mathbb{Z}$. Then $\Ext^1_{\mathbb{Z}}(\mathbb{Z}/\langle q\rangle,E)\cong E/qE$  by \cite[P.267, property (G)]{FS15}. So $p^n\Ext^1_{\mathbb{Z}}(\mathbb{Z}/\langle q\rangle,E)\cong (p^nE+qE)/qE\not=0$ for any $n\geq 0$. Thus $E$ is not   $s$-absolutely pure. Setting $E'=\bigoplus\limits_{n=1}^{\infty}E_n$, one can similarly shows that  $E'$ is also not   $s$-absolutely pure.
\end{example}

Recall from \cite{zusapm} that an $R$-module $M$ is said to be uniformly $S$-absolutely pure provided that there is $s\in S$ such that for any finitely presented $R$-module $N$ we always have $s\Ext_R^1(N,M)=0$.

Recall that a multiplicative subset $S$ of $R$ is said to satisfy the maximal multiple condition if there exists a $t\in S$ such that $s|t$ for each $s\in S.$

Trivially, every uniformly $S$-absolutely pure  is $s$-absolutely pure. The converse is true when  $S$  satisfies the maximal multiple condition.

\begin{proposition} Suppose $S$ is a multiplicative subset of a ring $R$ satisfying the maximal multiple condition. Then every $s$-absolutely pure module is uniformly $S$-absolutely pure.
\end{proposition}
\begin{proof} Let $M$  be an $s$-absolutely pure module. Then for any finitely presented $R$-module $N$, there exists $s_N\in S$ such that $s_N\Ext_R^1(N,M)=0$. Since $S$  satisfies the maximal multiple condition, there exists a $t\in S$ such that $s|t$ for each $s\in S.$  So $t\Ext_R^1(N,M)=0$ for each finitely presented $R$-module $N$. Hence, $N$ is uniformly $S$-absolutely pure.
\end{proof}

However, every $s$-absolutely pure module is not uniformly $S$-absolutely pure in general.
\begin{example}\cite[Example 3.9]{zusapm}
Let $R=\mathbb{Z}$ be the ring of integers, $p$ a prime in $\mathbb{Z}$ and  $S=\{p^n|n\geq 0\}$. Let $J_p$ be the additive group of all $p$-adic integers. Then  $J_p$ is  $s$-absolutely pure but not  $u$-$S$-absolutely pure.
\end{example}
\begin{proof} Let $N$ be a finitely presented $R$-module. Then, by \cite[Chapter 3, Theorem 2.7]{FS15}, $N\cong \mathbb{Z}^n\oplus \bigoplus\limits_{i=1}^m(\mathbb{Z}^n/\langle p^i\rangle)^{n_i}\oplus T$,
 where  $T$ is a finitely generated torsion $S$-divisible torsion-module. Thus $$\Ext_{R}^1(N,J_p)\cong \bigoplus\limits_{i=1}^m\Ext_{R}^1(\mathbb{Z}^n/\langle p^i\rangle,J_p)\cong \bigoplus\limits_{i=1}^m  (J_p/p^i J_p)\cong \bigoplus\limits_{i=1}^m\mathbb{Z}^n/\langle p^i\rangle$$ by \cite[Chapter 9, section 3 (G)]{FS15} and \cite[Chapter 1, Exercise 3(10)]{FS15}. So $\Ext_{R}^1(N,J_p)$ is obviously   $u$-$S$-torsion. However, $J_p$ is not  uniformly $S$-injective by \cite[Theorem 4.5]{QKWCZ21}. So  $J_p$ is not  $u$-$S$-absolutely pure.
 \end{proof}

\section{Some new characterizations of $S$-coherent rings}
 In 2018, Bennis and Hajoui \cite{bh18} introduced the notions of $S$-finitely presented modules and $S$-coherent modules which can be seen as $S$-versions of finitely presented modules and coherent modules. An $R$-module $M$ is said to be \emph{$S$-finitely presented} provided  there exists
an exact sequence of $R$-modules $$0 \rightarrow K\rightarrow F\rightarrow M\rightarrow 0,$$ where $K$ is
$S$-finite and $F$ is  finitely generated free. An $R$-module $M$ is said to be \emph{$S$-coherent} if it is finitely generated and every finitely generated submodule of $M$ is $S$-finitely presented. A ring $R$ is called an \emph{$S$-coherent ring} if $R$ itself is an $S$-coherent $R$-module, that is, every finitely generated ideal of $R$ is $S$-finitely presented.

Recently, the authors \cite{qwz23} gave an $S$-version of Chase Theorem in terms of $S$-flat modules.

 \begin{theorem}\cite[Theorem 4.4]{qwz23} Let $R$ be an ring and $S$  a multiplicative closed set of $R$. Then the following assertions are equivalent:
\begin{enumerate}
    \item $R$ is an $S$-coherent ring;
    \item any product of flat $R$-modules is $S$-flat;
       \item any product of  projective $R$-modules is $S$-flat;
        \item any product of copies of  $R$ is $S$-flat.
\end{enumerate}
\end{theorem}

Next, we characterized $S$-coherent rings in terms of $s$-flat modules.
 \begin{theorem}\label{newchase} Let $R$ be an ring and $S$  a multiplicative closed set of $R$. Then the following assertions are equivalent:
	\begin{enumerate}
		\item $R$ is an $S$-coherent ring;
		\item any product of flat $R$-modules is $s$-flat;
		\item any product of  projective $R$-modules is $s$-flat;
		\item any product of copies of  $R$ is $s$-flat.
	\end{enumerate}
\end{theorem}
\begin{proof} We only need to show $(1)\Rightarrow (2)$. Let $\{F_i\}$ be a family of flat $R$-modules. Let $M$ be a finitely presented $R$-module. Then there is an exact sequence $$0\rightarrow N\rightarrow P\rightarrow M\rightarrow 0$$ with $P$ finitely generated projective. It follows by \cite[Theorem 7]{znotescoh}that  $N$ is $S$-finitely presented. Then	there exists
	an exact sequence of $R$-modules $$0 \rightarrow K\rightarrow Q\rightarrow N\rightarrow 0,$$ where $K$ is
	$S$-finite and $Q$ is  finitely generated projective.
	So there exists
	an exact sequence of $R$-modules $$0 \rightarrow K'\rightarrow K\rightarrow T\rightarrow 0,$$ where $K'$ is finitely generated and $T$ is uniformly $S$-torsion.
	
Consider the following commutative diagrams of exact sequences:
$$\xymatrix@R=20pt@C=25pt{
	0\ar[r]^{}&\Tor_1^R(M,\prod F_i)\ar[r]^{}\ar[d]^{\theta^1_M}&  N\otimes_R\prod F_i\ar[r]^{}\ar[d]^{\theta_N} & P\otimes_R\prod F_i\ar[d]^{\cong} &\\
	0\ar@{=}[r]^{}&\prod\Tor_1^R(M, F_i) \ar[r]^{}& \prod (N\otimes_RF_i) \ar[r]^{} &\prod (P\otimes_RF_i) & \\}$$
	
$$\xymatrix@R=20pt@C=25pt{
&K\otimes_R\prod F_i\ar[r]^{}\ar[d]^{\theta_K}&  Q\otimes_R\prod F_i\ar[r]^{}\ar[d]^{\cong} & N\otimes_R\prod F_i\ar[d]^{\theta_N}\ar[r]^{} &0\\
	0\ar[r]^{}&\prod (K\otimes_RF_i)\ar[r]^{}& \prod (Q\otimes_RF_i) \ar[r]^{} &\prod (N\otimes_RF_i)\ar[r]^{} &0 \\}$$	
and	
	$$\xymatrix@R=20pt@C=25pt{
	\Tor_1^R(T,\prod F_i)\ar[r]^{}	&K'\otimes_R\prod F_i\ar[r]^{}\ar[d]^{\theta_{K'}}&  K\otimes_R\prod F_i\ar[r]^{}\ar[d]^{\theta_K} & T\otimes_R\prod F_i\ar[d]^{\theta_T}\ar[r]^{} &0\\
		0\ar[r]^{}&\prod (K'\otimes_RF_i)\ar[r]^{}& \prod (K\otimes_RF_i) \ar[r]^{} &\prod (T\otimes_RF_i)\ar[r]^{} &0 .\\}$$	
The rest of this proof can be induced by the uniform $S$-version of five-lemma (see \cite[Theorem 1.2]{ztugdocr}). Indeed, 	
since $K'$ is finitely generated, $\theta_{K'}$ is an epimorphism. So  $\theta_K$ is a uniformly $S$-epimorphism.  Then $\theta_N$ is a uniformly $S$-isomorphism. So  $\theta^1_M$
is also a uniformly $S$-isomorphism. Hence $\Tor_1^R(M,\prod F_i)$ is uniformly $S$-torsion. Consequently, $\prod F_i$	is $s$-flat.
\end{proof}

It is well-known that a ring $R$ is coherent if and only if any pure quotient of  absolutely pure $R$-modules is absolutely pure, if and only if any direct limit of  absolutely pure $R$-modules is absolutely pure. Now, we give an $S$-version of this result.

\begin{theorem}\label{s-d-d-non}
Let $R$ be a ring. The following statements are equivalent:
\begin{enumerate}
    \item  $R$ is $S$-coherent.
    \item any $s$-pure quotient of  $s$-absolutely pure $R$-modules is $s$-absolutely pure;
    \item any pure quotient of  $s$-absolutely pure $R$-modules is $s$-absolutely pure;
      \item any pure quotient of  absolutely pure $R$-modules is $s$-absolutely pure;
   \item  any direct limit of  absolutely pure $R$-modules is $s$-absolutely pure.

\end{enumerate}
\end{theorem}
\begin{proof} $(1)\Rightarrow (2)$: Suppose  $R$ is an $S$-coherent ring. Let $B$ be an $s$-absolutely pure $R$-module. Let $C$ be an $s$-pure quotient of $B$ and $A$ be an  $s$-pure submodule of $B$. Let $P$ be a finitely generated projective $R$-module and $K$ finitely generated submodule of $P$. Let $i:K\rightarrow P$ be the natural embedding map,  and  $f_C:K\rightarrow C$ an $R$-homomorphism. We claim that there is an $R$-homomorphism $g:P\rightarrow C$ such that $sf_C=gi.$  Consider the following exact sequence
$0\rightarrow L\xrightarrow{i_L} P'\xrightarrow{\pi_{P'}} K\rightarrow 0$ with $P'$ finitely generated projective. Then we have the following exact sequence of commutative diagram:
$$\xymatrix@R=25pt@C=40pt{
	0\ar[r]^{}&L \ar[r]^{i_L}\ar[d]_{f_A} &P' \ar[d]^{f_B}\ar[r]^{} &\ar[d]^{f_C} K  \ar[r]^{}& 0 \\
	0\ar[r]&A\ar[r]^{i_A} &B\ar[r]^{\pi_B} & C \ar[r]^{}&0.\\ }$$
 Since $R$ is $S$-coherent, $L$ is $S$-finite, so there is an exact sequence $0\rightarrow K'\xrightarrow{i_{K'}} L\xrightarrow{\pi_L} T\rightarrow 0$ with $s_1T=0$ for some $s_1\in S$. Consider the following push-out:

$$\xymatrix@R=20pt@C=25pt{ & & 0\ar[d]&0\ar[d]&\\
	0 \ar[r]^{} & K'\ar@{=}[d]\ar[r]^{i_{K'}} & L\ar[d]_{i_L}\ar[r]&T\ar[d]\ar[r]^{} &  0\\
	0 \ar[r]^{} & K'\ar[r]^{} & P' \ar[d]\ar[r]^{} &X\ar[d]\ar[r]^{} &  0\\
&  & K \ar[d]\ar@{=}[r]^{} &K\ar[d] &  \\
	& & 0 &0 &\\}$$
Let $f_A:L\rightarrow A$ be an $R$-homomorphism. It follows by Proposition \ref{spsub} that $A$ is $s$-absolutely pure, there is an $R$-homomorphism $g_A:P'\rightarrow A$ such that $s_2f_Ai_{K'}=g_Ai_Li_{K'}$ for some $s_2\in S$. Hence $s_1s_2f_A=s_1g_Ai_L$. It follows by  \cite[Exercise 1.60]{fk16} that we have the following commutative diagram:
$$\xymatrix@R=25pt@C=40pt{
	0\ar[r]^{}&L \ar[r]^{i_L}\ar[d]_{s_1s_2f_A} &P'\ar@{.>}[ld]^{s_1g_A} \ar[d]^{s_1s_2f_B}\ar[r]^{} &\ar[d]^{s_1s_2f_C} K \ar@{.>}[ld]^{g_B} \ar[r]^{}& 0 \\
	0\ar[r]&A\ar[r]^{i_A} &B\ar[r]^{\pi_B} & C \ar[r]^{}&0,\\ }$$
satisfying $s_1s_2f_C=\pi_Bg_B$.
Since $B$ is $s$-absolutely pure,  there exists $R$-homomorphism $g'_B:P\rightarrow B$ such that the following diagram is commutative:
	$$\xymatrix@R=30pt@C=40pt{
	K\ar[r]^{i}\ar[d]_{s_3g_B}&P\ar@{.>}[ld]^{g'_B}\\
B&\\}$$
for some $s_3\in S$ by Theorem \ref{c-s-abp}.
Setting $s=s_1s_2s_3$, we have $sf_C=s_1s_2s_3f_C=s_3\pi_Bg_B=\pi_Bg'_Bi=gi$, where $g=\pi_Bg'_B$.  It follows by Theorem \ref{c-s-abp} that $C$ is an $s$-absolutely pure $R$-module.

	$(2)\Rightarrow (3)\Rightarrow (4)$: Trivial.
		
	$(4)\Rightarrow (5)$:	Let $\{M_i\}_{i\in\Gamma}$ be a direct system of absolutely pure $R$-modules. Then there is an pure exact sequence  $0\rightarrow K\rightarrow \bigoplus\limits_{i\in\Gamma}M_i\rightarrow {\lim\limits_{\longrightarrow _{i\in\Gamma}}}M_i\rightarrow 0$. Note that $\bigoplus\limits_{i\in\Gamma}M_i$ is absolutely pure, so is ${\lim\limits_{\longrightarrow _{i\in\Gamma}}}M_i$ by (2).

	$(5)\Rightarrow (1)$: Let $I$ be a finitely generated  ideal of $R$,  $\{M_i\}_{i\in\Gamma}$  a direct system of $R$-modules. Let $\alpha: I\rightarrow \lim\limits_{\longrightarrow }M_i$ be a homomorphism. For any $i\in \Gamma$, $E(M_i)$ is the injective envelope of $M_i$. Then $E(M_i)$ is absolutely pure. By (4), we have $s\Ext_R^1(R/I, \lim\limits_{\longrightarrow }M_i)=0$ for some $s\in S.$ So there exists an $R$-homomorphism $\beta:R\rightarrow \lim\limits_{\longrightarrow }E(M_i)$ such that the following  diagram commutes:
	$$\xymatrix@R=20pt@C=25pt{
		0\ar[r]^{}&I\ar[r]^{}\ar[d]_{s\alpha}&  R\ar[r]^{}\ar@{.>}[d]^{\beta} & R/I\ar@{.>}[d]^{}\ar[r]^{} &0\\
		0\ar[r]^{}&{\lim\limits_{\longrightarrow }}M_i \ar[r]^{}&{\lim\limits_{\longrightarrow }}E(M_i)  \ar[r]^{} &{\lim\limits_{\longrightarrow }}E(M_i)/M_i\ar[r]^{} & 0.\\}$$
	Thus, by \cite[Lemma 2.13]{gt}, there exists $j\in \Gamma$, such that $\beta$ can factor through $R\xrightarrow{\beta_j} E(M_j)$. Consider the following commutative diagram:
	$$\xymatrix@R=20pt@C=25pt{
		0\ar[r]^{}&I\ar[r]^{}\ar@{.>}[d]_{s\alpha_j}&  R\ar[r]^{}\ar[d]^{\beta_j} & R/I\ar@{.>}[d]^{}\ar[r]^{} &0\\
		0\ar[r]^{}&M_j \ar[r]^{}&E(M_j)  \ar[r]^{} &E(M_j)/M_j\ar[r]^{} & 0.\\}$$
	Since the composition $I\rightarrow R\rightarrow E(M_j)\rightarrow E(M_j)/M_j$ becomes to be $0$ in the direct limit, we can assume $I\rightarrow R\rightarrow E(M_j)$ can factor through some  $I\xrightarrow{s\alpha_j} M_j$. Thus $s\alpha$ can factor through $M_j$. Consequently, the natural homomorphism  $\lim\limits_{\longrightarrow } \Hom_R(I,M_i)\xrightarrow{\phi}  \Hom_R(I, \lim\limits_{\longrightarrow }M_i)$ is a $u$-$S$-epimorphism with respect to $s$. Now suppose  $\{M_i\}_{i\in\Gamma}$ is a direct system of finitely presented $R$-modules such that $\lim\limits_{\longrightarrow } M_i=I$. Then there exists $f\in \Hom_R(I,M_j)$ with $j\in \Gamma$ such that the identity map  $s\Id_I= \phi(u_j(f))$ where $u_j$ is the natural homomorphism $\Hom_R(I,M_j)\rightarrow \lim\limits_{\longrightarrow } \Hom_R(I,M_i)$. That is, we have the following commutative diagram: 	
	$$\xymatrix@R=10pt@C=40pt{
		I\ar[rd]^{f}\ar[dd]_{s\Id_I}&\\
		&M_i\ar[ld]^{\delta_i}\\
		I=\lim\limits_{\longrightarrow } M_i&\\}$$
	So 	$I$ is $u$-$S$-finitely presented.
	We claim that $I$ is $S$-finitely presented. Indeed,  since $I$ is finitely generated, there is an exact sequence $0\rightarrow L\rightarrow Q\rightarrow I\rightarrow 0$with $Q$ finitely generated free. It follows by \cite[Theorem 2.2(4)]{zuscoh} that $L$ is $S$-finite. Hence $R$ is an $S$-coherent ring.
\end{proof}

 It is well-known that  a ring $R$ is  coherent if and only if $\Hom_R(E,I)$ is  flat for any absolutely pure module $E$ and injective module $I$, if and only if $\Hom_R(\Hom_R(F,I_1),I_2)$ is flat  for any  flat module $F$     and any injective modules $I_1$ and $I_2$ (see \cite[Lemma 4]{CD93}).

\begin{theorem}\label{phi-coh-fp} Let $R$ be a ring and $S$ be a   multiplicative subset of $R$.  Then the following statements are equivalent:
	\begin{enumerate}
		\item $R$ is an  $S$-coherent ring;
		\item    $\Hom_R(E,I)$ is $s$-flat  for any    $s$-absolutely pure module $E$  and any injective module $I$;
		\item    if $I$  is an injective cogenerator, then   $\Hom_R(E,I)$ is $s$-flat  for any  $s$-absolutely pure module $E$;
		\item   $\Hom_R(\Hom_R(F,I_1),I_2)$ is $s$-flat  for any  $s$-flat module $F$     and any injective modules $I_1$ and $I_2$;
\item if $I_1$ and $I_2$ are injective cogenerators, then $\Hom_R(\Hom_R(F,I_1),I_2)$ is $s$-flat    for any   $s$-flat module $F$.
	\end{enumerate}
\end{theorem}
\begin{proof}
	$(2)\Rightarrow (3)$ and $(4)\Rightarrow (5)$: Trivial.
	
	$(2)\Rightarrow (4)$ and $(3)\Rightarrow (5)$: This follows from Proposition \ref{flat-FP-injective}.
	
	$(1)\Rightarrow (2)$: Suppose $R$ is a  $S$-coherent ring. Let $M$ be a finitely presented $R$-module. Then there is an exact sequence $$0\rightarrow Q/K\rightarrow P\rightarrow M\rightarrow 0$$ with $P,Q$ finitely generated projective. It follows by \cite[Theorem 7]{znotescoh} that  $Q/K$ is $S$-finitely presented. So $K$ is	$S$-finite.	And hence there exists	an exact sequence of $R$-modules $$0 \rightarrow T\rightarrow Q/K'\rightarrow Q/K\rightarrow 0,$$ where $K'$ is finitely generated and $T$ is uniformly $S$-torsion.
	
	 Consider the following commutative diagrams with exact rows ($(-,-)$ is instead of $\Hom_R(-,-)$):
	$$\xymatrix@R=20pt@C=15pt{
		(E,I)\otimes_R T\ar[r]^{} \ar[d]^{\psi^1_{T'}} &(E,I)\otimes_R Q/K' \ar[d]_{\psi_{Q/K'}}^{\cong}\ar[r]^{} & (E,I)\otimes_R Q/K \ar[d]_{\psi_{Q/K}}\ar[r]^{}&0\\
		((T,E),I)\ar[r]^{} & ((Q/K',E),I) \ar[r]^{} &((Q/K,E),I)\ar[r]^{} &0,\\ }$$
	and
	$$\xymatrix@R=20pt@C=15pt{
		0\ar[r]^{}&\Tor_1^R((E,I),M) \ar[r]^{}\ar[d]^{\psi^1_{M}} &(E,I)\otimes_R Q/K \ar[d]_{\psi_{Q/K}}\ar[r]^{} & (E,I)\otimes_R Q \ar[d]_{\psi_{Q}}^{\cong}\ar[r]^{} &(E,I)\otimes_R M \ar[d]_{\psi_{M}}^{\cong}\ar[r]^{} &0\\
		0\ar[r]&(\Ext_R^1(M,E),I) \ar[r]^{} &((Q/K,E),I)\ar[r]^{} & ((Q,E),I) \ar[r]^{} &((M,E),I)\ar[r]^{} &0\\ }$$
Note that $\psi_{Q/K'}$ is an isomorphism by  \cite[Proposition 8.14(1)]{hh} and \cite[Theorem 2]{ELMUT1969}. It follows by chasing diagram and  the uniform $S$-version of five-lemma (see \cite[Theorem 1.2]{ztugdocr}) that $\psi^1_{M}$ is a uniformly $S$-isomorphism. Since $E$ is $s$-absolutely pure, $\Ext_R^1(M,E)$ is $u$-$S$-torsion, so is $\Hom_R(\Ext_R^1(M,E),I)$. Hence $\Tor_1^R(\Hom_R(M,E),M)$ is $u$-$S$-torsion , and thus $\Hom_R(M,E)$ is $s$-flat by Proposition \ref{flat-FP-injective}.

$(5)\Rightarrow (1)$: Let $\{F_i\}$ be a family of flat $R$-module. We will show $\prod F_i$ is $s$-flat. Since $\oplus F_i$ is flat,  by assumption $$\Hom_R(\Hom_R(\oplus F_i,I_1),I_2)\cong \Hom_R(\prod\Hom_R(F_i,I_1),I_2)$$ is $s$-flat.  Note that $\oplus\Hom_R(F_i,I_1)$ is a pure submodule of $\prod\Hom_R(F_i,I_1)$. Then the natural epimorphism $$\Hom_R(\prod\Hom_R(F_i,I_1),I_2\rightarrow \Hom_R(\oplus\Hom_R(F_i,I_1),I_2)\rightarrow 0$$ splits. Hence, $\prod\Hom_R(\Hom_R(F_i,I_1),I_2)\cong \Hom_R(\oplus\Hom_R(F_i,I_1),I_2)$ is also $s$-flat. Since $\prod F_i$ is a pure submodule of $\prod \Hom_R(\Hom_R(F_i,I_1),I_2)$, we have  $\prod F_i$ is $s$-flat by Lemma \ref{spq-f}. It follows by Theorem \ref{newchase} that $R$ is $S$-coherent.
\end{proof}

\end{document}